\documentclass[14pt]{amsart}
\address{\newline{\normalsize Kavli IPMU (WPI), The University of Tokyo, 5\,-\,1\,-\,5 Kashiwanoha, Kashiwa, 277\,-\,8583, Japan}
\bigskip
\newline{\normalsize Moscow Institute of Physics and Technology, 9 Institutskiy per., Dolgoprudny,
 Moscow Region, 141701, Russia}
\bigskip
\newline{\it E-mail address}: ilya.karzhemanov@ipmu.jp}
\usepackage{amscd,amsthm,amsmath,amssymb}
\usepackage[dvips]{graphicx}
\usepackage[matrix,arrow]{xy}

\makeatletter\@addtoreset{equation}{section}\makeatother

\makeatletter\@addtoreset{subsection}{equation}\makeatother

\newtheorem{theorem}[equation]{Theorem}
\newtheorem{prop}[equation]{Proposition}
\newtheorem{lemma}[equation]{Lemma}
\newtheorem{cor}[equation]{Corollary}

\newtheorem{theorem-definition}[equation]{Theorem-definition}

\theoremstyle{definition}

\theoremstyle{remark}
\newtheorem{remark}[equation]{Remark}

\newcommand{\com}{\mathbb{C}}

\newcommand{\p}{\mathbb{P}}

\newcommand{\ra}{\mathbb{Q}}
\newcommand{\cel}{\mathbb{Z}}

\newcommand{\aut}{\textbf{Aut}}

\newcommand{\fie}{{\bf k}}
\newcommand{\gal}{\mathrm{Gal}}
\newcommand{\map}{\longrightarrow}

\thanks{{\it MS 2010 classification}: 14E05, 14D06, 14M20}

\begin{document}

\title{On endomorphisms of hypersurfaces}

\thanks{}

\author{Ilya Karzhemanov}

\begin{abstract}
For any prime $p\ge 5$, we show that generic hypersurface
$X_p\subset\p^p$ defined over $\ra$ admits a non-trivial rational
dominant self-map of degree $>1$, defined over $\overline{\ra}$. A
simple arithmetic application of this fact is also given.
\end{abstract}

\sloppy

\maketitle

\bigskip

\section{Introduction}
\label{section:intro}

\refstepcounter{equation}
\subsection{}

Let $X$ be an algebraic variety (smooth, projective, over a field
of characteristic $0$). Then two groups of symmetries of $X$,
namely $\text{Aut}(X)$ (biregular automorphisms) and
$\text{Bir}(X)$ (birational ones), come for free. The question on
whether these groups differ is classical and very important. For
instance, the property $\text{Aut}(X) = \text{Bir}(X)$, signifying
\emph{birational (super)rigidity} of $X$, is an obstruction for
$X$ to be rational (such $X$ had been studied in numerous papers
including \cite{isk-man}, \cite{pux}, \cite{corti}, \cite{isk-pux}
and \cite{kol}).

Typically one has $\text{Aut}(X) = \{\mathrm{id}\}$ however. In
this regard, it is natural to consider another (more general)
class of symmetries of $X$, namely $\text{End}(X)$, consisting of
\emph{rational dominant} endomorphisms of $X$. Then one may ask
(what seems to be a folklore) whether $\text{End}(X) =
\text{Bir}(X)$? (Note that the latter property is an obstruction
for $X$ to be unirational.) This question gets immediate answer
(`yes') when $X$ --- with $\text{Pic}(X) = \cel$ at least --- is
of general type (as the endomorphisms preserve the spaces
$H^0(X,mK_X)$ for all $m\in\cel$). On the other hand, already in
the Calabi-Yau case things are not that straightforward; although
still one gets $\text{End}(X) = \text{Bir}(X)$ when $X$ is a
general K3 surface for example (see \cite{chen}).

Presently, we would like to treat rationally connected $X$, namely
$X := X_N\subset\p^N$ being a hypersurface of degree $N$ (see
\cite{pux}, \cite{isk}, \cite{che-ilya} for some other aspects of
the geometry of these $X_N$). Recall that according to \cite{de-f}
every $X_N,N\ge 4$, is non-rational, having $\text{Bir}(X_N) =
\text{Aut}(X_N)$.

Our main result is

\begin{theorem}
\label{theorem:main} For any prime $p\ge 5$, (Zariski) general
$\ra$-hypersurface $X_p\subset\p^p$ admits an endomorphism
$f_{X_p}$ of degree $>1$, defined over $\overline{\ra}$. More
precisely, $f_{X_p}$ has a field of definition --- $\fie$ below,
--- same for all $X_p$.
\end{theorem}

Thus the ``non-regular" geometry of $X = X_p$ is pretty much
fruitful (compare with results in \cite{a-r-v}, \cite{bea},
\cite{ps}, accounting for \emph{regular} self-maps of $X$). At the
same time, the proof of Theorem~\ref{theorem:main} is not
effective, and it would be interesting to describe (a part of)
$\text{End}(X)$ explicitly. For example, what is the
$\text{End}(X)$-action on the universal Chow group of $X$ (cf.
\cite{claire}, \cite{colliot-et-al})? It is also plausible to get
rid of the degree/ground field/dimension assumption in the
formulation of our result. Say, can one take \emph{any} Fano
manifold $X$ in place of $X_p$, or at least any (composite)
integer $N$ instead of $p$?

\refstepcounter{equation}
\subsection{}

In order to prove Theorem~\ref{theorem:main} we first relate $X_p$
to those hypersurfaces that have lots of endomorphisms. The latter
are $X_d\subset\p^N$, defined over a given field, that happen to
be unirational over this field, provided that $N\gg d$ is
sufficiently large (see {\ref{subsection:pre-1}} below for the
precise statements).

Next we employ the \emph{degree formula} from \cite{zay} (cf.
\cite{mer}, \cite{kollar}, \cite{hau}). Recall that this formula
relates the connective K-theory classes of two algebraic varieties
$X$ and $Y$ admitting some rational map $f: X \dashrightarrow Y$.
However, in order to get something fruitful in this way (e.\,g. to
show that $f$ with $\deg f = 0$ does not exist) one has to
consider $X,Y$, etc. to be defined over an algebraically
\emph{non-closed} field $\fie$, and (roughly speaking) to have no
points over $\fie$. This is the reason for the degree (resp.
dimension/ground field) restriction in Theorem~\ref{theorem:main}.

\begin{remark}
\label{remark:main-r} One of the basic obstructions for applying
our method to arbitrary degree $d$ (Fano) hypersurfaces
$X_d\subset\p^N$ is that those usually have points over the field
of definition (cf. \cite{birch}). There may also be no such nice
congruence relation as in Corollary~\ref{theorem:zay-cor-1} below.
But still, once we find (sufficiently many?) endomorphisms of
$X_p\subset\p^p$ which preserve some projective subspace
$\Gamma\subset X_p$, one may hope (by projecting from $\Gamma$) to
obtain non-trivial endomorphisms of $X_d$ for some (all?) $d\le
p-1$.
\end{remark}

The principal part of our arguments relies on (a part of) the main
result in \cite{colliot} which asserts the existence of $X = X_p =
Y$ as indicated above. It is then not hard to derive
Theorem~\ref{theorem:main} for the given $X_p$ and the general
case follows easily (see Section~\ref{section:const} for details).
Again we indicate that the initial $X_p$ is very special. (It is
defined over a field $\fie$ of cohomological dimension $\le 1$ and
does not contain points over the extensions of $\fie$ whose
degrees are coprime with $p$.)

\refstepcounter{equation}
\subsection{}

The next result was motivated by the paper \cite{abr}:

\begin{cor}
\label{theorem:main-cor} Let $X_p\subset\p^p$ be as in
Theorem~\ref{theorem:main}. Then there exists a possibly larger
number field $K\supseteq\ra$ such that Zariski closure of the
$f_{X_p}$-orbit of the set $X_p(K)$ of $K$-rational points on
$X_p$ has dimension $\ge 2$.
\end{cor}

One may consider Corollary~\ref{theorem:main-cor} as a
generalization of \cite[Theorem 1.4]{har-yu}. Yet, unfortunately,
our conclusion is weaker and it would be interesting to establish
potential density of the set $X_p(\ra)$ in $X_p$ (e.\,g. by
refining the arguments of Section~\ref{section:arith} below).

\bigskip

\thanks{{\bf Acknowledgments.}
I would like to thank F. Bogomolov and K. Zainoulline for their
interest and fruitful conversations. The results of this paper
have been discussed with K. Oguiso, S. Okawa, and B. Poonen during
my visits to Osaka University and MIT in Spring 2015; I am
grateful to these people for hospitality. Also comments by
anonymous referee have helped to improve the exposition. The work
was supported by World Premier International Research Initiative
(WPI), MEXT, Japan, and Grant-in-Aid for Scientific Research
(26887009) from Japan Mathematical Society (Kakenhi).

\bigskip

\section{Preliminaries}
\label{section:prelim}

\refstepcounter{equation}
\subsection{}
\label{subsection:pre-1}

Fix some integers $d$ and $r \ge 1$. Recall that every smooth
hypersurface $X_d\subset\p^N$ of degree $d$ contains a projective
subspace $\Lambda\simeq\p^r$ when $N\gg 1$ (see \cite{h-m-p},
\cite{kol-et-al}). More precisely, $X_d$ corresponds to (generic)
point in the incidence subvariety
$$
Z := \left\{(X_d,\Lambda)\ \vert \ \Lambda\subset
X_d\right\}\subset
\p^{\scriptscriptstyle\binom{N+d}{d}}\times\text{Gr}(r+1,N+1),
$$
with \emph{dominant} projections $Z\map\p,\,\text{Gr}$.

Further, if $\Lambda$ is given over an arbitrary field
$\fie_0\subset\com$ by equations $l_{N-r+1}=\ldots=l_N=0$ for some
linear forms $l_i$, then $X_d$ can be chosen to be defined over
$\fie:=\fie_0(\sqrt{-1})$. Indeed, any $X_d$ passing through
$\Lambda$ has the defining equation $\displaystyle\sum_{i=N-r+1}^N
\phi_il_i = 0$, for some (varying) forms $\phi_i$ of degree $d -
1$. Hence, since the set of $\fie$-points is obviously dense in
$\p^{\scriptscriptstyle\binom{N+d-1}{d-1}}$ (w.\,r.\,t. the
complex analytic topology), one can approximate $\phi_i$ by such
$\fie$-forms of degree $d - 1$ that $X_d$ remains smooth.

Thus, since one can always find generic $\Lambda\in\text{Gr}(r,N)$
to be defined over $\fie$, the above discussion provides a smooth
hypersurface $X_d$ and a projective subspace $\Lambda\subset X_d$,
both over $\fie$, for any given $d,r$ and sufficiently large
$N=N(d,r)$. From \cite[Corollary 3.7]{h-m-p} we obtain

\begin{theorem}[Harris, Mazur, Pandharipande]
\label{theorem:x-d-uni} $X_d$ is $\fie$-unirational.
\end{theorem}

\refstepcounter{equation}
\subsection{}

Let $X,Y$ be smooth projective geometrically irreducible
$\fie_0$-varieties of dimension $d=\dim X=\dim Y$. Assume that
there is a rational $\fie_0$-map $f: X \dashrightarrow Y$. The
degree $\deg f$ equals $0$ if $f$ is non-dominant; otherwise $\deg
f := [\fie_0(X):f^*\fie_0(Y)]$.

We recall the next result from \cite{zay}:

\begin{theorem}[Zainoulline]
\label{theorem:zay} $\chi(\mathcal{O}_X)\cdot\tau_{d-1}\equiv\deg
f\cdot\chi(\mathcal{O}_Y)\cdot\tau_{d-1}\mod n_Y$, where
$\chi(\cdot)$ is the Euler characteristic,
$$
\tau_{d-1} := \prod_{p\
\mathrm{prime}}p^{~\left[\frac{d-1}{p-1}\right]}
$$
is the $(d-1)$-st Todd number, and $n_Y$ is the g.\,c.\,d. of
degrees of all closed points on $Y$.
\end{theorem}

An immediate consequence of Theorem~\ref{theorem:zay} is

\begin{cor}[{cf. \cite[Example 6.4]{zay}}]
\label{theorem:zay-cor-1} In the previous notation, if $n_X = p$,
$X = Y = X_p\subset\p^p$ (i.\,e. $d = p-1$ and $N = p$) for some
prime $p \ge 3$, and (more generally) $f$ is a $\fie$-map, then
$\deg f
> 0$. (Such $X$ is called \emph{incompressible} over
$\fie$.)\footnote{~Existence of similar hypersurfaces X of general
type, as suggested by \cite[Example 6.4]{zay}, is not clear in the
current setting because the condition $n_X = p$ need not be
satisfied (compare with Proposition~\ref{theorem:c-y-c-1-dim}
below).}
\end{cor}

\begin{proof}
Regard $X$ as a hypersurface over $\fie$. Then, since
$[\fie:\fie_0]\le 2$, we get $n_X = p$ again. The claim now
follows from Theorem~\ref{theorem:zay} (with $\fie_0$ replaced by
$\fie$) and the fact that $\chi(\mathcal{O}_X) = 1 =
(\tau_{p-2},p)$.
\end{proof}

Existence of $X = X_p$ as in Corollary~\ref{theorem:zay-cor-1} is
guaranteed by the following result (see \cite[Theorem
8]{colliot}):

\begin{prop}[Colliot-Th\'el\`ene] \label{theorem:c-y-c-1-dim} For every prime $p\ge 5$, there is a
smooth (over $\overline{\ra}$) hypersurface $X\subset\p^p$, given
(over $\ra$) by the equation
$$
x_1^p + lx_2^p + \ldots + l^{p-1}x_p^p - \alpha x_0^p = 0
$$
for some integers $l,\alpha$, so that $n_X = p$.
\end{prop}

\begin{remark}
\label{remark:x-p-in-p} Let $t$ be a transcendental parameter,
$X_d\subset\p^N$ some $\ra(t)$-hypersurface, and $d = N = p$. One
may assume (e.\,g. by identifying $X_d$ with an appropriate pencil
of degree $d = p$ hypersurfaces) that this $X_d$ specializes via
$t \mapsto 0$ to $X$ from Proposition~\ref{theorem:c-y-c-1-dim}.
It follows then that generic such $X_d$ also satisfies $n_{X_d} =
p$. (Here ``generic hypersurface'' means a point in a Zariski open
subset of the $\ra(t)$-variety $\p^{\scriptscriptstyle\binom{N +
d}{d}}$ parameterizing all hypersurfaces of degree $d$ in $\p^N$
defined over the field $\ra(t)$.) Indeed, by definition of $n_Y$
(see Theorem~\ref{theorem:zay}) and specialization $t \mapsto 0$
one finds that $n_{X_d}$ divides $p = n_X$, and once $n_{X_d} = 1$
we get (by the same reasoning) that $n_X = 1$ as well, a
contradiction.
\end{remark}

\bigskip

\section{Proof of Theorem~\ref{theorem:main}}
\label{section:const}

We keep on with notation of Section~\ref{section:prelim}.

\refstepcounter{equation}
\subsection{}
\label{subsection:pre-adaddf}

Let $X$ be as in Proposition~\ref{theorem:c-y-c-1-dim}. Consider a
cone $\widehat{X}\subset\p^N$ over $X$ of sufficiently large
dimension and a family $\mathcal{X}$ of degree $p$ hypersurfaces
($\subset\p^N$) over $\ra$ that smooths out the singularities of
$\widehat{X}$. We may regard (the general fiber of) $\mathcal{X}$
as a smooth $\fie_0$-hypersurface of degree $p$ in $\p^N$ for some
purely transcendental field $\fie_0\supset\ra$.

Next, using the (dominant) projections $Z\map\p,\,\text{Gr}$ one
can see by the same argument as in {\ref{subsection:pre-1}} that
the set of all $\fie$-points $(X_p,\Lambda)\in Z$ is dense on $Z$
in the complex analytic topology, and so $\mathcal{X}$ can be
approximated by $\fie$-hypersurfaces $X_p$ (for $\mathcal{X},X_p$
treated as points on $\p$, with $\fie\subset\com$). The
hypersurfaces $\mathcal{X}$ and $X_p$ can actually be put on an
affine line $\mathbb{A}^1_{\fie}\subset\p$ in such a way that the
preimage $\widetilde{\mathbb{A}^1_{\fie}}\subset Z$ of
$\mathbb{A}^1_{\fie}$ has all fibers = some projective spaces and
generic fiber isomorphic to $\Lambda$ (apply the reasoning from
{\ref{subsection:pre-1}} to this family over $\mathbb{A}^1_{\fie}$
considered as a degree $p$ hypersurface over $\fie_0(t)$).

\begin{lemma}
\label{theorem:1-par-fam-ex} In the previous setting, the
hypersurface $\mathcal{X}$ contains a projective subspace
$\subset\p^N$, isomorphic to $\Lambda$ and defined over $\fie$.
\end{lemma}

\begin{proof}
It suffices to show that there exists a $\fie$-point on $Z$ whose
projection to $\p$ coincides with $\mathcal{X}$. Indeed, the
preceding natural projection
$\widetilde{\mathbb{A}^1_{\fie}}\map\mathbb{A}^1_{\fie}$ is
$\gal(\bar{\fie}\slash\fie)$-equivariant, hence its fiber over
$\mathcal{X}$ is a projective $\fie$-space.
\end{proof}

Without loss of generality we will assume that $\fie_0 = \ra(t)$
in what follows.

\begin{lemma}
\label{theorem:good-end} There exists a $\fie$-endomorphism $f:
\mathcal{X} \dashrightarrow \mathcal{X}$ such that $\deg f>1$.
\end{lemma}

\begin{proof}
It follows from Lemma~\ref{theorem:1-par-fam-ex} and
Theorem~\ref{theorem:x-d-uni} (applied to $X_d := \mathcal{X}$)
that $\mathcal{X}$ is $\fie$-unirational. This yields two rational
dominant $\fie$-maps $\phi:\p^{N-1}\dashrightarrow\mathcal{X}$ and
$\psi:\mathcal{X}\dashrightarrow\p^{N-1}$. We may assume one of
$\phi,\psi$ to have degree $>1$. Then it remains to take $f :=
\phi\circ\psi$.
\end{proof}

\refstepcounter{equation}
\subsection{}
\label{subsection:pre-adad}

Let $f$ be as in Lemma~\ref{theorem:good-end}. Note that $f$ is
induced by a rational self-map of
$\p^N\supset\mathcal{X},\widehat{X}$. Indeed, since $p \ge 5$ and
thus $\text{Pic}(\mathcal{X}) = \cel$ (Lefschetz), the map $f$ is
given by such a linear system on $\mathcal{X}$ that is obtained
via restriction of a linear system from $\p^N$.

Put $f_0$ to be the specialization of $f$ at the fiber
$\widehat{X}$ of the family $\mathcal{X}$ (recall that $\fie_0 =
\ra(t)$). More precisely, if $\mathcal{L}_t$ is a (movable) linear
system defining $f$, then its specialization $\mathcal{L}_0$ to
$\widehat{X}$ may acquire some divisorial components in the base
locus.\footnote{~The latter stems from the fact that $f$ may not
be defined (in codimension $2$) on $\widehat{X}$.} Subtracting all
these we arrive at a linear system which we set to define $f_0$.

\begin{lemma}
\label{theorem:f-0-not-sing} $f_0$ is not induced by a self-map of
$\text{Sing}(\widehat{X})$ (for an appropriate $f$). In other
words, if $x_0 = \ldots = x_p = 0$ are the equations of the
singular locus $\text{Sing}(\widehat{X})\subset\widehat{X}$, then
$f_0\ne\mathrm{id},0$ on the subspace $\p^p\subset\p^N$
complementary to $\text{Sing}(\widehat{X})$.
\end{lemma}

\begin{proof}
Let the notation be as in the proof of
Lemma~\ref{theorem:good-end}. One can choose a $\fie$-point
$o\in\mathcal{X}$ such that $\psi$ (resp. $\phi$) is unramified at
$o$ (resp. $\psi(o)$).

Take $f := \phi\circ\sigma\circ\psi$ for some
$\sigma\in\text{PGL}(N,\fie)$ (to be specified further) in such a
way that $f(o) = o$. Namely, since $\psi,\phi$ induce isomorphisms
of the tangent spaces $T_o = T_{\phi(\psi(o))},T_{\psi(o)}$, and
(adjoint of) $\sigma$ can act transitively on the $N$-tuples of
$\fie$-vectors in $T_{\psi(o)}$, we arrive at such $f$ whose
Jacobian $\text{Jac}_o(f)$ is a $\fie$-matrix with pairwise
distinct eigen values and all defined over $\ra$. Then,
considering $f_0$ as a rational self-map of $\p^N$ (cf. the
discussion at the beginning of {\ref{subsection:pre-adad}}), we
obtain that all non-zero eigen values of the matrix
$\text{Jac}_{o'}(f_0)$ are pairwise distinct as well. Here
$o'\in\widehat{X}$ is the specialization of $o$.\footnote{~Note
that $o'$ is not contained in the indeterminacy locus of $f_0$ for
the indicated $o$ and $\sigma$. Thus $f_0$ is defined at $o'$ and
satisfies $f_0(o') = o'$.}

Now suppose that $f_0$ \emph{is} induced by some rational
endomorphism of $\text{Sing}(\widehat{X})$. Again we regard $f_0$
as a self-map of $\p^N$. Then $f_0$ is regular at some point
$o'\in\p^N$, $f_0(o') = o'$ and the non-zero eigen values of
$\text{Jac}_{o'}(f_0)$ are all different. On the other hand, $f_0
= \text{id}$ on the subspace $\p^p\subset\p^N$ complementary to
$\text{Sing}(\widehat{X})$, a contradiction.

Same argument shows that $f_0$ has non-trivial components on
$\p^p$, i.\,e. $f_0 \ne 0$ there, which concludes the proof.
\end{proof}

Note that generic subspace $\p^p\subset\p^N$ cuts out a subvariety
on $\widehat{X}$ isomorphic $X$ (cf. the beginning of
{\ref{subsection:pre-adaddf}}). Identify $X$ with such a section
$\p^p \cap \widehat{X}$ so that the restriction $f_0\big\vert_X$
is a well-defined rational map. Let $f_X: X\dashrightarrow X$ be
the composition of $f_0\big\vert_X$ and the linear projection
$\widehat{X}\dashrightarrow X$ from $\text{Sing}(\widehat{X})$.

\begin{lemma}
\label{theorem:f-0-is-dom} $f_X\ne\mathrm{id}$ and is dominant.
\end{lemma}

\begin{proof}
Indeed, $f_0$ induces a non-identical map on $X$ by
Lemma~\ref{theorem:f-0-not-sing}, and it remains to apply
Corollary~\ref{theorem:zay-cor-1} to the $\ra(\sqrt{-1})$-map
$f_X$.
\end{proof}

\begin{remark}
\label{remark:some-dcos-gfg} For the last part of
Lemma~\ref{theorem:f-0-is-dom}, observe that the fact $\deg f_{X}
\ne 0$ is not immediate from the proof of
Lemma~\ref{theorem:f-0-not-sing}, and hence one needs an
additional argument (results of the second half of
Section~\ref{section:prelim} for instance) in order to proceed.
\end{remark}

Theorem~\ref{theorem:main} (for the given $X$) now follows from

\begin{prop}
\label{theorem:deg-more-1} $\deg f_X > 1$.
\end{prop}

\begin{proof}
Regard $\mathcal{X}$ as a (flat) family of hypersurfaces
$\mathcal{X}_t\subset\p^N\times t\subset \p^N\times\p^1$ (with
$\mathcal{X}_0 := \widehat{X}$). Let $\pi: \mathcal{X}\map\p^1$ be
the natural projection (so that $\pi^{-1}(t) = \mathcal{X}_t$).
Let also $\mathcal{L}_t$ be as in the second paragraph of
{\ref{subsection:pre-adad}}.
\begin{lemma}
\label{theorem:deg-more-1-asas} The linear system $\mathcal{L}_0$
is non-trivial on $\widehat{X}$ (unless $\deg f_X > 1$).
\end{lemma}

\begin{proof}
Suppose the contrary (i.\,e. $\mathcal{L}_0 = \{0\}$). Then the
$\pi$-map $f:\mathcal{X}\dashrightarrow\mathcal{X}$ has
indeterminacies along $\widehat{X}$. Resolve these by some
$\pi$-blow-up $\sigma:\mathcal{Y}\map\mathcal{X}$. Let
$\mathcal{Y}\stackrel{a}{\map}\mathcal{Z}\stackrel{b}{\map}\mathcal{X}$
be the Stein factorization of the resolved $f$. Here $a,b$ are
some $\pi$-morphisms, with $a$ inducing a birational isomorphism
between $\mathcal{Y}$ and $\mathcal{Z}$. Moreover, since $f$ is
not defined on $\widehat{X}$, the proper transform
$\sigma_*^{-1}\widehat{X}$ of $\widehat{X}$ on $\mathcal{Y}$
belongs to the exceptional locus of $a$. Composing further with
$b$ yields a rational self-map $\mu: \widehat{X} \dashrightarrow
\widehat{X}$ of degree $0$.

Notice that $\mu$ does not coincide with the projection
$\widehat{X} \dashrightarrow X$ from $\text{Sing}(\widehat{X})$
because $\deg b
> 1$ (cf. Lemma~\ref{theorem:good-end}) and $\mathcal{X}_0 = \widehat{X}$ is a non-multiple fiber of $\pi$ (so that $\widehat{X}$ is not a branching divisor of $b$).
Thus one may assume $X = \p^p \cap \widehat{X}$ intersects generic
fiber of $\mu$ at $>1$ points. Composing $\mu\big\vert_X$ with
$\widehat{X} \dashrightarrow X$ either gives a self-map $X
\dashrightarrow X$ of degree $>1$ (which we take for $f_X$), or
that $X$ is not incompressible, in contradiction with
Corollary~\ref{theorem:zay-cor-1}.
\end{proof}

Let $\mathcal{Y},\sigma,a,\ldots$ be as above. It follows from
Lemma~\ref{theorem:deg-more-1-asas} that $\widehat{X} =
\mathcal{X}_0$ and the scheme $\mathcal{Z}_0 := b^{-1}(0)$ are
birationally isomorphic via $a\circ\sigma^{-1}$ (recall that $b$
is \emph{finite}). In particular, $\mathcal{Z}_0$ is not a
ramification divisor of $b$, which implies that $\deg
b\big\vert_{\mathcal{Z}_t} = \deg b\big\vert_{\mathcal{Z}_0}$ for
all $t\in\com$ close to $0$.

Now, since $\deg b\big\vert_{\mathcal{Z}_t} > 1$ by construction,
we deduce that $\deg f_0 = \deg b\big\vert_{\mathcal{Z}_0} > 1$ as
well. The latter also gives $\deg f_X
> 1$. Indeed, otherwise restricting $\mathcal{L}_0$ to generic $X = \p^p \cap \widehat{X}$, we get $f_X\in\aut(X)$ according to
Lemma~\ref{theorem:f-0-is-dom} and \cite{de-f}. But then, since
$f_X$ is composed of $f_0\big\vert_X$ and projection $\widehat{X}
\dashrightarrow X$ to the base of the cone, we obtain that $f_0$
must be induced by some projective transformation of $\p^N \supset
\widehat{X}$. The latter obviously contradicts $\deg f_0 > 1$ and
the proof of Proposition~\ref{theorem:deg-more-1} is finished.
\end{proof}

To complete the proof of Theorem~\ref{theorem:main} we simply
apply Remark~\ref{remark:x-p-in-p} and the fact that all the
preceding arguments go verbatim with $\ra$ replaced by $\ra(t)$.
It remains to set $t := t_0 \in \ra$ --- a given general parameter
value --- to obtain hypersurfaces as in the statement of
Theorem~\ref{theorem:main}. Furthermore, in the generic setting
one can replace the argument with $f_X\in\aut(X)$ at the end of
the proof of Proposition~\ref{theorem:deg-more-1} by that with
$f_X = \text{id}$ (cf. \cite[Theorem 5]{ma-mon}), thus getting a
contradiction with Lemma~\ref{theorem:f-0-is-dom}.\footnote{~This
argument may be considered as another way to prove
Theorem~\ref{theorem:main}.}

\bigskip

\section{Proof of Corollary~\ref{theorem:main-cor}}
\label{section:arith}

\refstepcounter{equation}
\subsection{}

We proceed with applying the constructions of
Section~\ref{section:const} to study the arithmetics of
hypersurfaces $X_p\subset\p^p$ (the notation is as
earlier).\footnote{~Note that one can not specialize the
(abundance of) $\fie$-points on $\mathcal{X}$ to ``many"
$\ra$-points on the cone $\widehat{X}$ (as we did with (some of)
endomorphisms) because $n_{X} = p$ and so all the points on
$\widehat{X}$ we obtain this way are concentrated on
$\text{Sing}(\widehat{X})$.}

Let again $X$ be as in Proposition~\ref{theorem:c-y-c-1-dim}.

\begin{lemma}
\label{theorem:f-x-non-p} $f_X$ is non-periodic.
\end{lemma}

\begin{proof}
Indeed, otherwise we have $f_X^k = \text{id}$ for some $k$, so
that both $f_X,f_X^{k-1}$ are invertible. But this contradicts
Proposition~\ref{theorem:deg-more-1}.
\end{proof}

Fix $f_X$ as in Lemma~\ref{theorem:f-x-non-p}. Then after possibly
replacing $f_X$ by $f_X^k,k\gg 1$, we obtain a point $o\in X(K)$
such that $f_X(o) = o$ and $f_X$ is defined at $o$ (see
\cite{faxruddin}). We also have $\det\text{Jac}_o(f_X)\ne 0$
because $f_X$ is dominant.

\begin{lemma}
\label{theorem:inv-cub} There exists a $\fie$-cube
$\square_o\subset X$ (i.\,e. $\square_o$ is given by linear
inequalities with coefficients in $\fie$) centered at $o$ and
invariant under $f_X$.
\end{lemma}

\begin{proof}
Let $\square\subset X$ be some cube containing $o$ and defined
over $\ra$. Then the set $\displaystyle\bigcup_{k>0}X\setminus
f_X^k(\square)$ is not everywhere dense in $X$. Indeed, otherwise
there would exist a subsequence $f_X^{k_i}(\square)\to o$ for
$i\to\infty$, which implies that $o\in X(\ra)$ and contradicts
$n_X = p$. It remains to take $f_X$-invariant
$\square_o\subseteq\displaystyle\bigcap_{k>0}f_X^k(\square)$.
\end{proof}

It follows from Lemma~\ref{theorem:inv-cub} (and the implicit
function theorem) that the eigen values $\lambda_i$ of the matrix
$\text{Jac}_o(f_X)$ are algebraic numbers (from $K$), all having
norms $|\lambda_i| = 1,\,1\le i\le p-1$. Furthermore, making if
necessary a coordinate change on $\p^p\supset X$ of the form
$x_j\mapsto\alpha_j x_j,\,0\le j\le p$ (i.\,e. rescaling the
metric on $X$), for some $\alpha_j\in \ra^*$, we may assume
$\lambda_i$ to be algebraic integers (from $O_K$).

\begin{lemma}
\label{theorem:mult-indep} The matrix $\mathrm{Jac}_o(f_X)$ is
semi-simple. Moreover, there is $j < p-2$ such that $\lambda_1 =
\ldots = \lambda_j = \pm 1$, while
$\lambda_{j+1},\ldots,\lambda_{p-1}$ are multiplicatively
independent.
\end{lemma}

\begin{proof}
The first claim follows from the fact that $f_X$ linearizes on
$\square_o$. Now, if $\lambda_i = \pm 1$ (or equivalently
$\lambda_i\in\mathbb{R}$) for all $i$, then $f_X^2 = \text{id}$, a
contradiction.

Further, Dirichlet's unit theorem yields
$\lambda_{j+1}\ldots\lambda_{p-1} =
|\lambda_{j+1}\ldots\lambda_{p-1}| = 1,\,\lambda_j =
\lambda_{k}^{-1}$ as the only relations between
$\lambda_{j+1},\ldots,\lambda_{p-1}\in O_K^*$. In particular,
since $j < p - 1$, $\lambda_{j+1},\ldots,\lambda_{p-1}$ generate a
free subgroup $\subseteq \cel^{p-j-1}\subset O_K^*$.

Suppose that $j = p - 2$. Then
$\lambda_{j+1},\ldots,\lambda_{p-1}$ are all equal to
$\lambda_{j+1}^{\pm 1}$ say. On the other hand, characteristic
polynomial of $\text{Jac}_o(f_X)$ is defined over $\fie$ (by
construction of $\square_o$ in Lemma~\ref{theorem:inv-cub}), and
hence the minimal polynomial of $\lambda_{j+1}$ divides it. All
together, this implies that $\lambda^2_i = \pm 1$ for all $i\ge
j+1$, a contradiction.

Thus we get $j<p-2$ and the claim follows.
\end{proof}

The arguments in \cite[Section 2]{abr} and
Lemma~\ref{theorem:mult-indep} imply that Zariski closure of the
$f_X$-orbit of the locus $X(K)$ has dimension $\ge p - j - 1 > 1$.

Finally, to complete the proof of Corollary~\ref{theorem:main-cor}
one replaces $\ra$ by $\ra(t)$, as at the end of
Section~\ref{section:const}, and repeats the previous arguments.

\end{document}